\documentclass[11pt]{amsart}
\usepackage{amscd,amsxtra,amssymb,mathrsfs, bbm}
\usepackage{stmaryrd}
\usepackage{amscd,amsxtra,amssymb, bbm}
\usepackage[new]{old-arrows}

\usepackage{geometry}
        {\begin{quotation}\begin{center}\begin{em}}
        {\par\end{em}\end{center}\end{quotation}}

\newtheorem{theorem}{Theorem}[section]
\newtheorem{corollary}[theorem]{Corollary}
\newtheorem{lemma}[theorem]{Lemma}
\newtheorem{proposition}[theorem]{Proposition}

\theoremstyle{definition}
\newtheorem{definition}[theorem]{Definition}
\newtheorem{remark}[theorem]{Remark}

\DeclareMathOperator{\rad}{\mathsf{rad}}

\DeclareMathOperator{\Hom}{\mathsf{Hom}}

\DeclareMathOperator{\pd}{\mathsf{pd}}

\DeclareMathOperator{\Aut}{\mathsf{Aut}}
\DeclareMathOperator{\End}{\mathsf{End}}
\DeclareMathOperator{\Mat}{\mathsf{Mat}}

\def\bop{\bigoplus}

\newcommand{\kk}{\mathbbm{k}}

\newcommand{\NN}{\mathbb{N}}

\def\bop{\bigoplus}
\def\iff{if and only if }
\newcommand{\rightarrowdbl}{\longrightarrow\mathrel{\mkern-14mu}\rightarrow}
\newcommand{\twoheadarrow}{\rightarrow\mathrel{\mkern-14mu}\rightarrow}

\newcommand{\lar}{\longrightarrow}

\def\Arr{\Rightarrow}

\def\oA{\bar{A}}		\def\oH{\bar{H}}		\def\oM{\bar{M}}
\def\sb{\subseteq}

\def\kron#1#2{\xymatrix@C=2em{{#1}\ar@/^3pt/[r]\ar@/_3pt/[r]&{#2}}}

\usepackage{tikz}
\usetikzlibrary{calc, arrows, positioning, shapes, fit, matrix, decorations}
\usetikzlibrary{decorations.shapes, decorations.pathreplacing}

\tikzset{
  decorate with/.style={decorate,decoration={shape backgrounds,shape=#1,shape size=1.5mm}},
   deco/.style={decorate with=dart},
   ordi/.style={draw,-stealth,  thick},
   conj/.style={dashed, draw, thick},
   ve/.style={circle, draw, thick, fill=blue!20, inner sep=1pt, outer sep=2pt, minimum size=7pt},
    dot/.style={fill=blue!10,circle,draw, inner sep=1pt, minimum size=5pt},
  dv/.style={star,star points=5,
star point ratio=2, draw, thick, fill=green!20, inner sep=1pt,outer sep=2pt,minimum size=7pt}
}

\tikzset{
    tbl5 nodes/.style={
        rectangle,
        execute at begin node=$,
       execute at end node=$,
       fill=blue!5,
        align=center,
        text depth=0.5ex,
        text height=2ex,
        inner xsep=0pt,
        outer sep=0pt,
           },
    tbl5/.style={
        matrix of nodes,
        row sep=-\pgflinewidth,
        column sep=-\pgflinewidth,
        nodes={
            tbl5 nodes
        },
        execute at empty cell={\node[draw=none]{};}
    }
  }

\input xy
\xyoption{all}

\title[Some aspects of the theory of nodal orders]{Some aspects of the theory of nodal orders}

\author{Igor Burban}
\address{
Paderborn University
}
\email{burban@math.uni-paderborn.de}

\author{Yuriy Drozd}
\address{
Harvard University and Institute of Mathematics, National Academy of
Sciences of Ukraine}
\email{y.a.drozd@gmail.com}


\begin{document}

\begin{abstract}
In this paper, we elaborate ring theoretic properties of nodal orders. In particular, we prove that they are  closed under taking crossed  products with finite groups. 
\end{abstract}

\maketitle

\section{Introduction}
Nodal orders are non-commutative generalizations of the ring $\kk\llbracket x, y\rrbracket/(xy)$, where $\kk$ is a field. This class of rings was introduced by the second-named author in \cite{NodalFirst}.

\begin{definition}\label{D:Order}
Let $R$ be  an excellent  reduced  equidimensional ring of {Krull dimension one}  and
$K := \mathsf{Quot}(R)$ the corresponding total ring of fractions. 
An $R$--algebra $A$ is an $R$--\emph{order} if the following conditions are fulfilled:
\begin{itemize}
\item $A$ is a finitely generated torsion free $R$--module.
\item $A_K:= K \otimes_R A$ is a  semi--simple $K$--algebra, having finite length as a $K$--module.
\end{itemize}
A ring $A$ is an \emph{order} if 
 its center $R = Z(A)$ is a reduced excellent  equidimensional ring  of Krull dimension one, and $A$ is an $R$--order. If 
$K:= \mathsf{Quot}(R)$ then $A_K:= K \otimes_R A$ is called \emph{rational envelope} of the order $A$.
\end{definition}

\begin{definition}\label{D:NodalOrders}
An order $A$ is called \emph{nodal} if its center is semi--local  and there exists a \emph{hereditary} overorder $H \supseteq A$ such that the following conditions are satisfied.
\begin{itemize}
\item $J:=\rad(A) = \rad(H)$.
\item For any  finitely generated simple left $A$--module $U$ the length of  $H \otimes_A U$ over $A$ is at most two. 
\end{itemize}
\end{definition}
It is clear that an order  $A$ is nodal if and only if its radical completion $\widehat{A}$ is nodal. Moreover, given  a nodal order $A$,  a hereditary overorder $H$ from Definition \ref{D:NodalOrders} is 
 \emph{uniquely determined} and admits the following intrinsic description:
\begin{equation}\label{E:HereditaryCover}
H = \bigl\{x \in A_K \, \big| \, x J \subseteq J\bigr\} \cong \End_A(J),
\end{equation} where $J$ is viewed as a right $A$--module and $A_K$ is the rational envelope of $A$ (this is essentially a consequence of \cite[Theorem 39.11 and Corollary 39.12]{ReinerMO}).  For a nodal order $A$, the order
$H$ will be called \emph{hereditary cover} of $A$. Hereditary orders form a special subclass of nodal orders. 

Let $R = \kk\llbracket t\rrbracket$ be the algebra of formal power series, where $\kk$ is an algebraically closed field. It was shown in \cite{NodalFirst} that an $R$--order $A$ is representation wild (i.e.~the category of finite length left $A$--modules is representation wild) unless $A$ is nodal. On the other hand, proper nodal orders (i.e.~those orders which are not hereditary) are representation tame (hereditary orders are well-known to be representation discrete; see for instance \cite{DrozdHereditary}). 
 In our  joint work \cite{Nodal}, we  proved that  the derived category of the category of finite length modules over an nodal order $A$ 
has tame representation type.

\smallskip
\noindent
 Of great interest is the global version of nodal orders given by (tame) non-commutative projective nodal curves; see for instance  \cite{bd, bdnpdalcurves}.

\smallskip
\noindent
The main result of this work the following. 

\smallskip
\noindent
\textbf{Theorem}. Let $A$ be a nodal order and $G$ be a finite group acting 
on $A$ (this means that we are given a map 
$G \stackrel{\phi}\lar \Aut(A)$  and a normed two-cocycle $G \times G \stackrel{\omega}\lar A^\ast$ satisfying certain  compatibility conditions described below). If $|G|$ is invertible in $A$ then the crossed product ring $A\bigl[G, (\phi, \omega)\bigr]$ is again a nodal order.

\smallskip
\noindent
\emph{Acknowledgement}. The work of the first-named author was partially supported by the German Research Foundation
SFB-TRR 358/1 2023 -- 491392403.

\section{Reminder on crossed product rings}
\begin{definition}\label{D:GroupAction}
Let $A$ be a ring, $A^\ast$ be its group of units and $G$ be a finite group. Assume we are given a pair of maps
\begin{equation}
G \stackrel{\phi}\lar \Aut(A) \quad \mbox{\rm and} \quad G \times G \stackrel{\omega}\lar A^\ast
\end{equation}
satisfying the following conditions.
\begin{itemize}
\item $\phi_{g_1}\bigl(\phi_{g_2}(a)\bigr) = \omega_{g_1, g_2} \cdot  \bigl(\phi_{g_1 g_2}(a)\bigr) \cdot \omega_{g_1, g_2}^{-1}$ for all $g_1, g_2 \in G$ and $a \in A$. 
\item $\phi_e = \mathsf{Id}$ and $\omega_{g, e} = 1 = \omega_{e, g}$ for all $g \in G$, where $e \in G$ is the neutral element. 
\item $\phi_g\bigl(\omega_{g', g''}\bigr) \cdot \omega_{g, g' g''} = \omega_{g, g'} \cdot \omega_{g g', g''}$ for all $g, g', g'' \in G$. 
\end{itemize}
Then we say that $G$ acts on $A$ via $\phi$ with a system of factors $\omega$; see \cite{ReitenRiedtmann}. The corresponding crossed product  $A\bigl[G, (\phi, \omega)\bigr]$ is a free left $A$--module of rank $|G|$: 
\begin{equation}
A\bigl[G, (\phi, \omega)\bigr] = \Bigl\{\sum\limits_{g \in G} a_g[g]  \, \big| a_g \in A  \Bigr\}
\end{equation}
equipped with the product given by the rule
$$
a[f] \cdot b[g] := a \phi_f(b) \omega_{f, g} [fg] \; \mbox{\rm for any} \; a, b \in  
A \; \mbox{\rm and} \;  f, g \in G.
$$
It is not difficult to show that $A\bigl[G, (\phi, \omega)\bigr]$ is a unital ring, whose multiplicative unit element is $1[e]$; see \cite{ReitenRiedtmann}.
\end{definition}

\begin{remark} Let $G \times G \stackrel{\omega^\circ}\lar A^\ast$ be trivial, i.e.~$\omega^\circ_{g_1, g_2} = 1$ for all $g_1, g_2 \in G$. Then 
$G \stackrel{\phi}\lar \Aut(A)$ is a group homomorphism and 
$
A\bigl[G, (\phi, \omega^\circ)\bigr] = A\ast G
$
is called skew group ring of $A$. 

\smallskip
\noindent
If the ring $A$ is commutative, then  $G \stackrel{\phi}\lar \Aut(A)$ is automatically a  group homomorphism. 
\end{remark}

\begin{definition}\label{D:Strictlyseparable}
A ring extension $A \subseteq B$ is called \emph{strictly separable} if it satisfies the following properties.
\begin{enumerate}
\item The multiplication map $B \otimes_A B \stackrel{\mu}\lar B$ is a split epimorphism of $B$--bimodules. 
\item The inclusion $A \stackrel{\imath}\lar B$ is a split monomorphism of $A$--bimodules. 
\item $B$ is projective as a left $A$--module and flat as a right $A$--module. 
\end{enumerate}
\end{definition}

\begin{remark} The first condition of Definition \ref{D:Strictlyseparable} implies that $B$ is a direct summand of $B \otimes_A B$ viewed as an $B$--subbimodule. An equivalent characterization of the first condition is the following: there exists an element $w \in B \otimes_A B$ such that $\mu(w) = 1$ and $bw = wb$ for all $b \in B$. 
Similarly, the second condition of Definition \ref{D:Strictlyseparable} implies that $A$ is a direct summand of $B$ viewed as an $A$--subbimodule. 
\end{remark}

\smallskip
\noindent
For a ring $C$ and a left $C$--module $M$ we denote by $\pd_C(M)$ its  projective dimension. 

\begin{lemma}\label{L:Dimensions}
Let $A \subseteq B$ be a strictly separable ring extension. 
\begin{enumerate}
\item For any left $B$--module $M$ we have: $\pd_B(M) = \pd_A(M)$.
\item For any left $A$--module $N$ we have: $\pd_A(N) = \pd_B(B \otimes_A N)$.
\item The left global dimensions of $A$ and $B$ are equal. 
\end{enumerate}
\end{lemma}

\begin{proof} (1)
Let $M$ be any left $B$--module. Since any projective resolution of $M$ in $B-\mathsf{Mod}$ is also a projective resolution in $A-\mathsf{Mod}$, we have: $\pd_B(M) \ge  \pd_A(M)$. Because the functor $B \otimes_A \,-\,$ is exact and maps projective $A$--modules to projective $B$--modules, we have: $\pd_A(M) \ge  \pd_B(B \otimes_A M)$. Next, observe  that
$B \otimes_A M \cong (B \otimes_A B) \otimes_B M$. Since $B$ is a $B$--bimodule direct summand of $B \otimes_A B$, it follows that the $B$--module $M$ is a direct summand of $B \otimes_A M$. Hence, $\pd_B(B \otimes_A M) \ge \pd_B(M)$ and, a a consequence, $\pd_A(M) \ge  \pd_B(M)$. Hence, $\pd_B(M) = \pd_A(M)$, as asserted. 

\smallskip
\noindent
(2) Let $N$ be any left $A$--module. Using again the fact that the functor $B \otimes_A \,-\,$ is exact and maps projective $A$--modules to projective $B$--modules, we have:
$$
\pd_A(N) \ge \pd_B(B \otimes_A N) = \pd_A(B \otimes_A N) \ge \pd_A(N),
$$
where the last inequality follows from the fact that $A$ is a direct summand of $B$ viewed as $A$--subbimodule. Hence, $\pd_A(N) = \pd_B(B \otimes_A N)$, as asserted. 

\smallskip
\noindent
(3) This statement is a consequence of the previous two. 
\end{proof}

\begin{proposition}\label{P:Dimensions} Let $A$ be a ring and $G$ be a finite group acting on $A$ via a datum $(\phi, \omega)$, as in Definition \ref{D:GroupAction}. Assume that $n:= |G|$ is a unit element in $A$. Then the ring extension $A \subseteq A[G, (\phi, \omega)]$ is strictly separable. 
\end{proposition}
\begin{proof} It is clear that $A[G, (\phi, \omega)]$ is a finitely generated projective  $A$--module, viewed  as a left or as a right module. Next, the map
$$
A[G, (\phi, \omega)] \stackrel{\pi}\rightarrowdbl A, \; \sum\limits_{g \in G} a_g[g] \mapsto a_e,
$$
is a morphism of $A$--bimodules. It follows that the inclusion $A \stackrel{\imath}\lar  A[G, (\phi, \omega)]$ is split. Finally, consider the following element
$$
w:= \frac{1}{n} \sum\limits_{g \in G} \omega^{-1}_{g, g^{-1}} [g] \otimes [g^{-1}] \in B \otimes_A B. 
$$
It is clear that $\mu(w) = 1$. Next, for any $g \in G$ and $a \in A$ we have the following equalities in $B \otimes_A B$: 
$$
[g] \otimes [g^{-1}] a = [g] \otimes \phi_{g^{-1}}(a) [g^{-1}] = [g]  \phi_{g^{-1}}(a) \otimes [g^{-1}]  = \phi_g\bigl(\phi_{g^{-1}}(a)\bigr) [g] \otimes [g^{-1}].
$$
Since $\omega_{g, g^{-1}}^{-1} \phi_g\bigl(\phi_{g^{-1}}(a)\bigr)= \phi_e(a) \omega_{g, g^{-1}}^{-1} = a \omega_{g, g^{-1}}^{-1}$, we conclude that $w a = a w$ for any $a \in A$. 

Next, let $g, h \in G$. Then we have the following equalities in $B \otimes_A B$:
$$
\omega_{g, g^{-1}}^{-1} [g] \otimes [g^{-1}][h] = \omega_{g, g^{-1}}^{-1} [g] \omega_{g^{-1}, h} \otimes   [g^{-1} h] = \omega_{g, g^{-1}}^{-1}  \phi_g\bigl(\omega_{g^{-1}, h}\bigr) [g] \otimes   [g^{-1} h].
$$
Since $\phi_g\bigl(\omega_{g^{-1}, h}\bigr) \omega_{g, g^{-1} h} = \omega_{g, g^{-1}}$, we see that $\omega_{g, g^{-1}}^{-1}  \phi_g\bigl(\omega_{g^{-1}, h}\bigr) = \omega^{-1}_{g, g^{-1} h}$. As a consequence, we get: 
$$
w [h] = \frac{1}{n} \sum_{g \in G} \omega^{-1}_{g, g^{-1} h} [g] \otimes [g^{-1} h].
$$
In a similar way, we have: 
$$
[h] \omega^{-1}_{g, g^{-1}} [g] \otimes [g^{-1}] = \phi_h\bigl(\omega^{-1}_{g, g^{-1}}\bigr) \omega_{h, g} [hg] \otimes [g^{-1}].
$$
Since $\phi_h\bigl(\omega_{g, g^{-1}}\bigr) = \phi_h\bigl(\omega_{g, g^{-1}}\bigr) \omega_{h, e} = \omega_{h, g} \omega_{hg, g^{-1}}$, we obtain: 
$\phi_h\bigl(\omega_{g, g^{-1}}\bigr)^{-1} \omega_{h, g} = \omega^{-1}_{hg, g^{-1}}$. It follows that 
$$
[h] w = \frac{1}{n} \sum_{g \in G} \omega^{-1}_{hg, g^{-1}}  [hg] \otimes [g^{-1}] = w [h]
$$
for any $h \in G$. Thus, we have shown that $b w = w  b$ for any $b \in A[G, (\phi, \omega)]$. 
\end{proof}

\begin{corollary}\label{C:GlobalDimensions} Assume that the assumptions of Proposition \ref{P:Dimensions} are fulfilled. Then Lemma \ref{L:Dimensions} implies that $A$ is semi-simple if and only if $A[G, (\phi, \omega)]$ is semi-simple. Analogously, $A$ is left (respectively, right) hereditary if and only if $A[G, (\phi, \omega)]$ is left (respectively, right) hereditary. 
\end{corollary}

\begin{remark} Proposition \ref{P:Dimensions} and Corollary \ref{C:GlobalDimensions}
are  essentially due to Reiten and Riedtmann; see \cite[Theorem 1.3]{ReitenRiedtmann}. Our exposition slightly deviates from the one of \cite{ReitenRiedtmann}.
\end{remark}

\begin{theorem}\label{T:InvariantAction} Let $A$ be a ring and $G$ be a finite group acting on $A$ via a datum $(\phi, \omega)$ as above. Let $P$ be $G$--invariant left $A$--module, i.e.~for any $g \in G$ we have an automorphism of abelian groups
$P \stackrel{\alpha_g}\lar P$ such that $\alpha_g(a p) = \phi_g(a) \alpha_g(p)$ for any $a \in A$ and $p \in P$. 
\begin{enumerate}
\item[(i)] There is an action $(\psi, \xi)$ of $G$ on the ring $\End_A(P)$ given by the following rules.
\begin{enumerate}
\item $G \stackrel{\psi}\lar \Aut\bigl(\End_A(P)\bigr), g \mapsto (\rho \stackrel{\psi_g}\mapsto \alpha_g \rho\alpha_g^{-1})$. 
\item $G \times G \stackrel{\xi}\lar \Aut_A(P), (f, g) \mapsto \xi_{f, g} := \lambda_{\omega_{f, g}^{-1}} \, \alpha_f \alpha_g \alpha_{fg}^{-1}$, where $P \stackrel{\lambda_b} \lar P$ is given by the formula $\lambda_b(p) = bp$ for any $b \in A$ and $p \in P$. 
\end{enumerate}
\item[(ii)] Moreover, the  map
\begin{equation}\label{E:RingHomom}
\End_A(P)\bigl[G, (\psi, \xi)] \stackrel{\Phi}\lar \End_{A[G, (\phi, \omega)]}\bigl(A\bigl[G, (\phi, \omega)\bigr] \otimes_A P \bigr)
\end{equation}
assigning to an element $\rho \{g\} \in \End_A(P)\bigl[G, (\psi, \xi)]$ the endomorphism $$
A\bigl[G, (\phi, \omega)\bigr] \otimes_A P  \lar A\bigl[G, (\phi, \omega)\bigr] \otimes_A P, \;  [h] \otimes p \mapsto [h] [g]^{-1} \otimes \rho\bigl(\alpha_g(p)\bigr)
$$
is a ring isomorphism. 
\end{enumerate}
\end{theorem}
\begin{proof}  Let $f, g \in G$. It is not difficult to check  that $\xi_{f, g}: P \lar P$ is $A$--linear. Next, it is obvious that $\psi_e = \mathsf{Id}$ and $\xi_{g, e} = 1 = \xi_{e, g}$ for all $g \in G$. Another straightforward computation shows that 
$
\xi_{f, g} \cdot \bigl(\psi_{fg}(\rho)\bigr) = \psi_f\bigl(\psi_g(\rho)\bigr) \cdot \xi_{f, g}
$
for all $f, g \in G$ and $\rho \in \End_A(P)$. 
The fact that 
$
\psi_f\bigl(\xi_{g, h}\bigr) \cdot \xi_{f, gh} = \xi_{f, g} \cdot \xi_{fg, h}
$
for any $f, g, h \in G$ can also be checked by a lengthy but completely straightforward verification. 

It is again straightforward that the map $\Phi$ given by (\ref{E:RingHomom}) is indeed a ring homomorphism. To show that $\Phi$ is bijective, consider the following commutative diagram of abelian groups:
\begin{equation*}
\begin{array}{c}
\xymatrix{
\End_A(P)\bigl[G, (\psi, \xi)\bigr] \ar[r]^-{\Phi} \ar[rd]_-{\Psi} &  \End_{A[G, (\phi, \omega)]}\bigl(A\bigl[G, (\phi, \omega)\bigr] \otimes_A P \bigr) \ar[d]^-{\mathsf{can}} \\
& \Hom_{A}\bigl(P, A\bigl[G, (\phi, \omega)\bigr] \otimes_A P \bigr) 
}
\end{array}
\end{equation*}
where $\mathsf{can}$ is the adjunction isomorphism and $\Psi$ is the map given by the formula 
$$
\Psi\Bigl(\sum\limits_{g \in G} \rho_g \{g\}\Bigr)(p) = 
\sum\limits_{g \in G} [g]^{-1} \otimes \rho_g\bigl(\alpha_g(p)\bigr)
$$
for any $p \in P$.  For a fixed $g \in G$ let $\vartheta_g := \rho_g \alpha_g$. Then for any $a \in A$ and $p \in P$ we have: $\vartheta_g(ap) = \phi_g(a) \vartheta_g(p)$. Note that we have an isomorphism of $A$--modules
$$
A\bigl[G, (\phi, \omega)\bigr] \otimes_A P \cong \bigoplus\limits_{g \in G} P_g,
$$
where $P_g = P$ as an abelian group, whereas the $A$--module structure on $P_g$ is given by the formula $a \ast p := \phi_g(a) \circ p$ (here,   $\circ$ is the action of $A$ on $P$). It follows that for any $\vartheta \in \Hom_{A}\bigl(P, A\bigl[G, (\phi, \omega)\bigr] \otimes_A P \bigr)$ one can find a unique family $\left\{\rho_g\right\}_{g \in G}$ of endomorphisms of $P$ such that $\Psi\bigl(\sum_{g \in G} \rho_g \{g\}\bigr) = \vartheta$. Hence, $\Psi$ is bijective and, as a consequence, $\Phi$ is bijective, too. 
\end{proof}

\section{Nodal pairs}\label{S:SectionNodal}
Let $A$ be a ring and $J = \rad(A)$ be its Jacobson radical. 
  Recall that a ring $A$ is called \emph{semilocal} if $\oA:= A/J$ is artinian. If, moreover, $A/J$ is a direct product 
  of skewfields, we call $A$ \emph{basic}. 
   Also, if $M$ is an $A$--module, we denote
  $\oM := M/JM$. For an $A$--module $M$ we denote by $\mu_A(M)$ the minimal number of generators of $M$ 
  and by $\ell_A(M)$ the length of $M$.   For a pair of semilocal rings $A\subseteq  H$ we denote 
  $$\ell^*(A,H)=\sup\bigl\{\ell_A(H\otimes_A U) \, \big|\,  U \text{ is a simple $A$--module}\bigr\}.$$

  \begin{definition}\label{nod} 
  Let $A\subseteq H$ be a pair of semilocal rings.
  \begin{enumerate}
  \item  This pair is called \emph{Backstr\"om} if $\rad(A)=\rad(H)$. 
  \item  It is called \emph{nodal} if it is Backstr\"om and  $\ell^*(A,H)\le2$.
  \item  If $H$ is left hereditary and this pair is Backstr\"om (nodal), $A$ is called a (left) \emph{Backstr\"om ring} (respectively,
  a (left) \emph{nodal ring}).
  \end{enumerate}
  \end{definition}
 
\smallskip
\noindent  
  The following fact is obvious.
   \begin{proposition}\label{modrad} 
   If a pair $A\sb H$ is nodal, so is the pair $\oA\sb\oH$. Conversely, if this pair is Backstr\"om and the pair $\oA\sb\oH$ is nodal,
   the pair $A\sb H$ is nodal too.
   \end{proposition}
   
   \begin{lemma}\label{morita} 
    Let $A\sb H$ be a Backstr\"om pair, $P$ be a progenerator in the category of right $A$-modules, $A'=\End_{A}(P)$,
    $P'=H\otimes_A P$ and $H'=\End_{H}(P')$. Then
    \begin{enumerate}
    \item  $P'$ is  a progenerator in the category of right $H$-modules.
    \item  $A'\sb H'$ is a Backstr\"om pair.
    \item  $\ell^*(A',H')=\ell^*(A,H)$.
    \end{enumerate}
     In particular, the properties of a ring ``to be Backstr\"om'' and ``to be nodal'' are Morita invariant.
    \end{lemma} 
     \begin{proof}
   See \cite{DrozdZembyk} (proof of Proposition \,3.1). 
     \end{proof}
         Under the conditions of this lemma we say that the pair $A'\sb H'$ is \emph{Morita equivalent} to the pair $A\sb H$.

\smallskip
\noindent         
We shall need the following combinatorial statement.          
          
      \begin{lemma}\label{number} For $n \in \NN$ let $B = \left(\begin{array}{ccc}
      b_{11} & \dots & b_{1n} \\ \vdots & \ddots & \vdots \\ b_{n1} & \dots & b_{nn} \end{array}\right) \in \Mat_{n \times n}(\NN_0)$ be such that
\begin{itemize}
\item $b_{ii} > 0$ for all $1 \le i \le n$.
\item For any $1 \le i \ne j \le n$ we have: $b_{ij} \ne 0$ if and only if $b_{ji} \ne 0$. 
\end{itemize}      
 For      $\vec{a} = \left(\begin{array}{c} a_1 \\ \vdots \\ a_n \end{array}\right)\in \NN^n$ and $\vec{a}' = B \vec{a}$ consider the following assertions:
     \begin{enumerate}
     \item  $a'_i \le 2a_i$ for all $1 \le i \le n$.
      \smallskip
     \item  For each $1 \le i \le n$ the following condition holds:  either $b_{ii}\le2$ and $b_{ij}=0$ for all $j\ne i$ or there is a unique index $1 \le i'\ne i \le n$ such that \\
        $a_i=a_{i'},\  b_{ii}=b_{ii'}=b_{i'i}=b_{i'i'}=1$  and $b_{ij}=0$ for all  $j\notin\{i,i'\}$.
        \smallskip
     \item  $\sum_{j=1}^nb_{ij}\le2$ for all $i$.     
     \end{enumerate}
     $(1)\Leftrightarrow(2)\Arr(3)$ and, if all $a_i$ are equal, $(3)\Arr(1)$.
      \end{lemma}
       \begin{proof}
       Obviously, $(2)\Arr(1)$, $(2)\Arr(3)$ and, if all $a_i$ are equal, $(3)\Arr(1)$.
       
       $(1)\Arr(2)$  is proved by induction on $n$. The case $n=1$ is trivial. Obviously, always $b_{ii}\le2$. If there is an index $i$ such that
       $a_{ij}=0$ for all $j\ne i$, just apply the induction to the set of all indices except $i$. Suppose there are no such indices, $a_i$ is the
       smallest and $b_{ii'}\ne0$ for some $i'\ne i$. Then $a_{i'}=a_i$, $b_{ii}=b_{ii'}=1$ and $b_{ij}=0$ if $j\notin\{i,i'\}$. As also $b_{i'i}\ne0$, we get
       $b_{i'i'}=b_{i'i}=1$ and $b_{i'j}=0$ if $j\notin\{i,i'\}$. Now apply induction to the set of all indices except $i$ and $i'$.
       \end{proof}
     
    \begin{theorem}\label{2gen} 
        Let $A\subseteq  H$ be a B\"ackstr\"om  pair, $U_1,U_2,\dots,U_n$ be all non-isomorphic  simple $A$-modules and $V_i=H\otimes_A U_i$ for $1 \le i \le n$.
        Consider the following assertions:
    \begin{enumerate}
    \item      $\mu_A(H)\le2$, i.e.~there exists a surjective $A$--linear map $A^{{2}}  \twoheadarrow  H$. 
    \item  For each $1 \le i \le n$, either $V_i\cong U_i^{b_i}$ with  $b_i\le2$, or there is a unique index
    $i'$ such that $V_i\cong  V_{i'}\cong  U_i\oplus U_{i'}$.
    \item  $A\subseteq H$ is a nodal pair, i.e. $\ell_A(V_i) \le 2$ for all $1 \le i \le n$. 
    \end{enumerate}
    $(1)\Arr(2)\Leftrightarrow(3)$ and, if $A$ is basic, $(3)\Arr(1)$.
       \end{theorem}
     \begin{proof} The implication $(2)\Arr(3)$ is obvious.
     
\smallskip
\noindent     
     Due to Proposition~\ref{modrad}, we can without loss of generality assume that  $A$ and $H$ are semi-simple. There exits a collection of primitive orthogonal idempotents $e_1, \dots, e_n \in A$ such that  $U_i\cong  Ae_i$ and $V_i\cong He_i$
     for all $1 \le i \le n$. 
     
\smallskip
\noindent     
     Let $A \cong \bop_{i=1}^nU_i^{a_i}$ and $V_i\cong  \bop_{i=1}^nU_j^{b_{ij}}$.
  Then 
     $$H\cong  \bop_{j=1}^n V_j^{a_j} \cong  \bop_{i=1}^nU_i^{\sum_{j=1}^nb_{ij}a_j}.$$ 
     Note that $b_{ij}\ne0$  \iff $\Hom_A(U_i,V_j)\cong  e_i H e_j \ne 0$. In particular, $b_{ii}>0$ for all $1 \le i \le n$. Moreover, if $e_iH e_j\ne0$, 
     then also $e_jHe_i\ne0$, since $H$ is semi-simple. It follows that $b_{ij} \in \NN_0$ and $b_{ij} = 0$ if and only if $b_{ji} = 0$. 
     So, the matrix $B = (b_{ij})$ satisfies the assumption of Lemma~\ref{number}. 
     
\smallskip
\noindent    
The first assertion  $\mu_A(H) \le 2$ is true if and only if $\sum\limits_{j = 1}^n b_{ij} a_j \le 2 a_i$ for all $1 \le i \le n$.            
The third assertion  means that $\sum_{j=1}^nb_{ij}\le2$ for all $i$. Finally,   $A$ being basic means that $a_i=1$ for all $1 \le i \le n$.

By Lemma~\ref{number}, we get an implication $(1)\Arr(3)$ and, if $A$ is basic, all three statements  are 
     equivalent. Since assertions (2) and (3) are invariant under Morita equivalences (see \cite{DrozdZembyk}), $(2)\Leftrightarrow(3)$ is true in a full generality. 
     \end{proof}

\begin{corollary}\label{C:Key} Let $A\subseteq  H$ be a B\"ackstr\"om  pair. 
        Consider the following assertions:
    \begin{enumerate}
    \item      $\mu_A(H)\le2$. 
    \item  $A\subseteq  H$ is a nodal pair.
    \end{enumerate}
   Then we have: $(1)\Arr(2)$ and $(1)\Leftrightarrow (2)$ provided  $A$ is basic. 
\end{corollary}

\section{Crossed products of nodal orders}
Let $A$ be an order (see Definition \ref{D:Order}) and $G$ be a finite group of order $n$ acting on $A$ via a datum $(\phi, \omega)$ as in Definition \ref{D:GroupAction}. We assume that $n$ is invertible viewed as an element of $A$. We make the following observations.

\begin{enumerate}
\item Let $R$ be the center of $A$. First note that $n^{-1}$ belongs to $R$. Next, for any $g \in G$ we have: $\phi_g(R) = R$. Moreover, $\phi_{g_1}\bigl(\phi_{g_2}(r)\bigr) = \phi_{g_1 g_2}(r)$ for any $g_1, g_2 \in G$ and $r \in R$. The ring of invariants $R^G$ is again an excellent reduced equidimensional ring of Krull dimension one. Moreover, the ring extension $R^G \subseteq R$ is finite. These statements can be shown in a similar way as for instance \cite[Theorem 4.1]{BurbanDrozdSurface}.

\item Let $K$ be the total rings of quotients of $R$ and $A_K$ be the rational hull of $A$. Then the action of $G$ on $A$ extends to an action on $A_K$: 
$$
A_K \stackrel{\widetilde{\phi}_g}\lar A_K, \; \frac{a}{b} \mapsto \frac{\phi_g(a)}{\phi_g(b)}.
$$
We see that $G$ acts on $A_K$ via the datum $(\widetilde{\phi}, \omega)$.

\item According to  Corollary \ref{C:GlobalDimensions}, the ring 
$A_K\bigl[G, (\widetilde\phi, \omega)\bigr]$ is semi-simple. Next, $R^G$ is contained in the center of the ring $A_K\bigl[G, (\phi, \omega)\bigr]$ and the extension $R^G \subset A\bigl[G, (\phi, \omega)\bigr]$ is finite. Next, $A\bigl[G, (\phi, \omega)\bigr]$ is torsion free viewed as a module over $R^G$ and its rational hull is $A_K\bigl[G, (\phi, \omega)\bigr]$. It follows, that the crossed product $A\bigl[G, (\phi, \omega)\bigr]$ is again an order. 

\item Now assume that $R$ is semi-local. Then $A$ is semi-local, too. Recall that the Jacobson radical $J$ of $A$ has the following characterization:
$$
J = \bigl\{x \in A \,\big|\, 1 + ax b \in A^\ast \; \mbox{for all} \; a, b \in A \bigr\}.
$$
It is clear that $\phi_g(J) = J$ for all $g \in G$. Let $\bar{A} := A/J$. It follows that for any $g \in G$ we have an induced ring automorphism 
$$\bar{A} \stackrel{\bar\phi_g}\lar \bar{A}, \; \bar{a} \mapsto \overline{\phi_g(a)}.
$$
Let $G \times G \stackrel{\bar\omega}\lar \bar{A}^\ast$ be given by the composition of $\omega$ and the natural group homomorphism  $A^\ast \lar 
\bar{A}^\ast$. It is clear that $G$ acts on $\bar{A}$ via the datum 
$(\bar\phi, \bar\omega)$. By Corollary \ref{C:GlobalDimensions}, the ring 
$\bar{A}\bigl[G, (\bar\phi, \bar\omega)\bigr]$ is semi-simple. It follows that $J\bigl[G, (\phi, \omega)\bigr]$ is the Jacobson radical of the order
$A\bigl[G, (\phi, \omega)\bigr]$.

\item Consider the  overorder $H$ of $A$ given by the formula (\ref{E:HereditaryCover}). It follows that $\widetilde{\phi_g}(H) = H$ for all $g \in G$. Hence, the group $G$ acts on $H$ via the datum $(\widetilde{\phi}, \omega)$ and the crossed product  
 $H\bigl[G, (\widetilde\phi, \omega)\bigr]$ is an overoder of 
$A\bigl[G, (\phi, \omega)\bigr]$. Moreover, if $H$ is hereditary then 
$H\bigl[G, (\widetilde\phi, \omega)\bigr]$ is hereditary, too. 
\end{enumerate}

\begin{theorem}\label{T:Main} Let $A$ be a nodal order and $G$ be a finite group of order $n$ acting on $A$ via a datum $(\phi, \omega)$. If $n$ is invertible in $A$ then the order $A\bigl[G, (\phi, \omega)\bigr]$ is again nodal. 
\end{theorem}

\begin{proof} A semi-local order is nodal if and only if its radical completion is nodal. Therefore we may without loss  of generality assume that $A$ is complete. Let us first suppose  that $A$ is basic. Let $H$ be the hereditary cover of $A$ (see (\ref{E:HereditaryCover}) for a description of $H$). By Corollary \ref{C:Key} we have: $\mu_A(H) \le 2$. It follows that $\mu_{A\bigl[G, (\phi, \omega)\bigr]}\bigl(H\bigl[G, (\widetilde\phi, \omega)\bigr]\bigr) \le 2$, too. Next, $J\bigl[G, (\phi, \omega)\bigr]$ is the common radical of $A\bigl[G, (\phi, \omega)\bigr]$ and $H\bigl[G, (\widetilde\phi, \omega)\bigr]$, hence $A\bigl[G, (\phi, \omega)\bigr] \subseteq H\bigl[G, (\widetilde\phi, \omega)\bigr]$ is a B\"ackstr\"om pair. By Corollary \ref{C:Key}, the order $A\bigl[G, (\phi, \omega)\bigr]$ is nodal. 

Let $A$ be now an arbitrary nodal order. Let $P$ be a projective generator of the category of finitely generated left $A$--modules, whose direct summands are pairwise non-isomorphic. Then $A' := \End_A(P)$ is  a basic nodal order;  see \cite[Proposition 1.3]{DrozdZembyk}. The $A$--module $P$ is invariant in the sense of Theorem \ref{T:InvariantAction}.  Hence, $A'$ admits an action of $G$ via a datum $(\psi, \chi)$ such that
$A'\bigl[G, (\psi, \xi)] \cong \End_{A[G, (\phi, \omega)]}\bigl(A\bigl[G, (\phi, \omega)\bigr] \otimes_A P \bigr)$. By the previous step we know that $A'\bigl[G, (\psi, \xi)]$ is a nodal order. Hence, the opposite order 
$\bigl(A'\bigl[G, (\psi, \xi)]\bigr)^\circ$ is nodal as well.  On the other hand, $A\bigl[G, (\phi, \omega)\bigr] \otimes_A P$ is a projective generator in the category of  finitely generated left 
$A\bigl[G, (\phi, \omega)\bigr]$--modules. Hence, $A\bigl[G, (\phi, \omega)\bigr]$ and $\bigl(A'\bigl[G, (\psi, \xi)])^\circ$ are Morita equivalent. Applying again \cite[Proposition 1.3]{DrozdZembyk}, we conclude that $A\bigl[G, (\phi, \omega)\bigr]$ is nodal, as asserted. 
\end{proof}

\end{document}